\documentclass[draft,12pt]{amsart}

\usepackage{amscd,amsmath,amsthm,amssymb}
\usepackage[left]{lineno}
\usepackage{color}
\usepackage{stmaryrd}
\usepackage[utf8]{inputenc}
\usepackage{epstopdf}
\usepackage{graphicx}
\usepackage{xcolor}

\usepackage{tikz}

\definecolor{verylight}{gray}{0.97}
\definecolor{light}{gray}{0.9}
\definecolor{medium}{gray}{0.85}
\definecolor{dark}{gray}{0.6}

\def\NZQ{\mathbb}               

\def\KK{{\NZQ K}}

\def\frk{\mathfrak}               

\def\mm{{\frk m}}

\def\KK{{\NZQ K}}

\renewcommand{\qedsymbol}{$\square$} 

\def\G{{\mathcal G}}

\def\ab{{\mathbf a}}

\def\0b{{\mathbf 0}}

\def\reg{{\mathbf reg}}
\def\height{\operatorname{ht}}
\def\depth{\operatorname{depth}}
\def\dstab{\operatorname{dstab}}

\def\opn#1#2{\def#1{\operatorname{#2}}} 
\opn\chara{char} \opn\length{\ell} \opn\pd{pd} \opn\rk{rk}
\opn\projdim{proj\,dim} \opn\injdim{inj\,dim} \opn\rank{rank}
\opn\depth{depth} \opn\grade{grade} \opn\height{height}
\opn\embdim{emb\,dim} \opn\codim{codim}
\opn\first{first} \opn\last{last}

\opn\Tr{Tr} \opn\bigrank{big\,rank}
\opn\superheight{superheight}\opn\lcm{lcm}
\opn\trdeg{tr\,deg}
	\opn\reg{reg} \opn\lreg{lreg} \opn\ini{in} \opn\lpd{lpd}
	\opn\size{size} \opn\sdepth{sdepth}
	\opn\link{link}\opn\fdepth{fdepth}\opn\lex{lex}
	\opn\tr{tr}
	\opn\type{type}
	\opn\gap{gap}
	\opn\arithdeg{arith-deg}
	\opn\HS{HS}
	\opn\GL{GL}
	\opn\div{div} \opn\Div{Div} \opn\cl{cl} \opn\Cl{Cl}
	\opn\Spec{Spec} \opn\Supp{Supp} \opn\supp{supp} \opn\Sing{Sing}
	\opn\Ass{Ass} \opn\Min{Min}\opn\Mon{Mon}
	\opn\Ann{Ann} \opn\Rad{Rad} \opn\Soc{Soc}\opn\Deg{Deg}
    \opn\odd{odd}
    \opn\even{even}
	\opn\Im{Im} \opn\Ker{Ker} \opn\Coker{Coker} \opn\Am{Am}
	\opn\Hom{Hom} \opn\Tor{Tor} \opn\Ext{Ext} \opn\End{End}
	\opn\Aut{Aut} \opn\id{id}
	
	\opn\nat{nat}
	\opn\pff{pf}
	\opn\Pf{Pf} \opn\GL{GL} \opn\SL{SL} \opn\mod{mod} \opn\ord{ord}
	\opn\Gin{Gin} \opn\Hilb{Hilb}\opn\sort{sort}
	\opn\PF{PF}\opn\Ap{Ap}
	\opn\mult{mult}
	\opn\bight{bight}
	\opn\aff{aff}
    \opn\esupp{esupp}
	\opn\relint{relint} \opn\st{st}
	\opn\lk{lk} \opn\cn{cn} \opn\core{core} \opn\vol{vol}  \opn\inp{inp} \opn\nilpot{nilpot}
	\opn\link{link} \opn\star{star}\opn\lex{lex}\opn\set{set}
	\opn\width{wd}
	\opn\Fr{F}
	\opn\QF{QF}
	\opn\G{G}
	\opn\type{type}\opn\res{res}
	\opn\conv{conv}
	\opn\Ind{Ind}
	\opn\gr{gr}
	
	\def\pot#1#2{#1[\kern-0.28ex[#2]\kern-0.28ex]}

	\opn\dirlim{\underrightarrow{\lim}}
	\opn\inivlim{\underleftarrow{\lim}}
	\let\to=\rightarrow
	
	\def\Implies{\ifmmode\Longrightarrow \else
		\unskip${}\Longrightarrow{}$\ignorespaces\fi}
	\def\implies{\ifmmode\Rightarrow \else
		\unskip${}\Rightarrow{}$\ignorespaces\fi}
	\def\iff{\ifmmode\Longleftrightarrow \else
		\unskip${}\Longleftrightarrow{}$\ignorespaces\fi}

	\let\:=\colon
	\newtheorem{Theorem}{Theorem}[section]
	\newtheorem{Lemma}[Theorem]{Lemma}

	\newtheorem{Remark}[Theorem]{Remark}
	
	\newtheorem{Example}[Theorem]{Example}
	
	\newtheorem{Definition}[Theorem]{Definition}

	\let\epsilon\varepsilon
	\let\kappa=\varkappa
	\def\qed{\ifhmode\textqed\fi
		\ifmmode\ifinner\quad\qedsymbol\else\dispqed\fi\fi}
	\def\textqed{\unskip\nobreak\penalty50
		\hskip2em\hbox{}\nobreak\hfil\qedsymbol
		\parfillskip=0pt \finalhyphendemerits=0}
	\def\dispqed{\rlap{\qquad\qedsymbol}}
	
	\opn\dis{dis}
	\def\pnt{{\raise0.5mm\hbox{\large\bf.}}}
	
	\opn\Lex{Lex}
	
\begin{document}
		\title {Depth of  powers of the  edge ideal of an increasing weighted path}
\date{\today}

	\author {Jiaxin Li}
\address{School of Mathematical Sciences, Soochow University, Suzhou 215006, P. R. China}
\email{lijiaxinworking@163.com}
		
\author {Dancheng Lu}
\address{School of Mathematical Sciences, Soochow University, Suzhou 215006, P. R. China}
\email{ludancheng@suda.edu.cn}

\author{Thanh Vu}
\address{Institute of Mathematics, VAST, 18 Hoang Quoc Viet, Hanoi, Vietnam}
\email{vuqthanh@gmail.com}
		
    \author{Guangjun Zhu$^{\ast}$}
    \address{ School of Mathematical Sciences, Soochow University, Suzhou 215006, P. R. China}
\email{zhuguangjun@suda.edu.cn}

\thanks{ $^{\ast}$ Corresponding author.}
	


		
		

		\begin{abstract}
	We provide exact formulas for the depth of the quotient ring of powers of the edge ideal of an increasing weighted path.
		\end{abstract}
		\setcounter{tocdepth}{1}
		
		\maketitle

	\section{Introduction}

Let $G$ be a finite simple graph with vertex set $V(G)$ and edge set $E(G)$. Suppose that ${\bf w}\colon E(G)\to \mathbb{Z}_{>0}$ is a weight function defined on the edges of $G$, and we denote the pair $(G,{\bf w})$ by $G_{\bf w}$. We call $G_{\bf w}$ an \emph{edge-weighted graph} with $G$ as its underlying graph.
For an edge-weighted graph $G_{\bf w}$ with vertex set $V(G)=\{x_1,\ldots,x_n\}$, its \emph{edge ideal} was first introduced in \cite{PS}. This ideal is a monomial ideal in the polynomial ring $S=\KK[x_1,\ldots,x_n]$ over a field $\KK$, which is defined by 
\[
I(G_{\bf w})
=\bigl( (x_ix_j)^{{\bf w}(x_ix_j)} \mid x_ix_j \in E(G)\bigr).
\]

If there exists an edge $e \in E(G)$ such that ${\bf w}(e)\ge 2$, then $G_{\bf w}$ is said to be \emph{nontrivially weighted}, and $e$ is referred to as a \emph{nontrivially weighted edge}. Otherwise, $G_{\bf w}$ is \emph{trivially weighted}, and every edge of $G_{\bf w}$ is a \emph{trivially weighted edge}. In this trivial case, $I(G_{\bf w})$ coincides with the classical edge ideal of the underlying graph $G$,  an ideal that has been extensively studied in the literature \cite{BC,FM,HHV,M,MTV}.

For a homogeneous ideal $I \subset S$, Brodmann \cite{B} showed that, for $t \gg 0$, the depth $\depth(S/I^t)$ becomes eventually constant. The smallest integer $t_0$ such that 
\[
\depth(S/I^t)=\depth(S/I^{t_0}) \text{\  for all } t \ge t_0
\]
 is called the \emph{index of depth stability} of $I$ and is denoted by $\dstab(I)$.
Trung \cite{T} determined the limit depth for arbitrary edge ideals of graphs and described the index of depth stability for certain classes of graphs. More recently, Lam, Trung, and Trung \cite{LTT} provided a complete combinatorial characterization of the index of depth stability for edge ideals of arbitrary graphs.

It is also worth noting that, although the index of depth stability of edge ideals is now well understood, the exact values of $\depth(S/I^t)$ for intermediate powers of edge ideals have only recently been computed for a limited number of graph classes; see, for example, \cite{BC,HHV,MTV}.

Turning to edge ideals of edge-weighted graphs, far less is known regarding both the limit depth and the index of depth stability. One notable result in this direction was recently established by Hien, Li, Trung, and Zhu \cite{HLTZ} for the case of increasing weighted trees.

The depth of powers of edge ideals of edge-weighted graphs has also been computed for some special classes of graphs; see, for example, \cite{LVZ,ZDCL,ZLCY}. Notably, the limit depth remains unknown even for edge ideals of arbitrary edge-weighted paths.

In this paper, we take a further step toward understanding the depth of powers of edge ideals of arbitrary edge-weighted paths by focusing on the case of increasing weighted paths. First, we introduce some notation for convenience. Let $P_{\bf w}^{n+1}$ denote an edge-weighted path with vertex set $V(P_{\bf w}^{n+1})=\{x_1,\ldots,x_{n+1}\}$ and edge set 
\[
E(P_{\bf w}^{n+1})=\{e_i \mid e_i := x_ix_{i+1}\text{ for } 1 \le i \le n\}.
\]
We say that $P_{\bf w}^{n+1}$ is an \emph{increasing weighted path} if ${\bf w}(e_i) \le {\bf w}(e_{i+1})$ for all $i$.

Given a weighted path $P_{\bf w}^{n+1}$, let
 \[
\Delta = \{\, i \in [n-2] \mid {\bf w}(e_i) = {\bf w}(e_{i+1}) \,\}.
\]
 The set $\Delta$ can be decomposed into maximal blocks 
 \[
 \Delta = \delta_1 \cup \cdots \cup \delta_p,
 \]
  where each $\delta_i$ consists of consecutive indices and satisfies
   \[
  \min(\delta_{i+1}) - \max(\delta_i) \ge 2.
  \]
A block $\delta$ is said to be of type $j$ (for $j = 0,1,2$) if $|\delta| \equiv j \pmod{3}$. Adjacent blocks of type $1$ and type $2$ can be \emph{glued} together when the distance between them is exactly $2$. Using this gluing operation, $\Delta$ can be decomposed as 
\[
	\Delta = \Delta_1 \cup \cdots \cup \Delta_q,
\]
 where each $\Delta_i$ is either a block of type $0$ or a maximal extended block formed by gluing together type $1$ and type $2$ blocks.

For $i = 1,\ldots,q$, let $t_i$ and $s_i$ denote the numbers of type $1$ and type $2$ blocks in $\Delta_i$, respectively, and let $r_i \in \{0,1\}$ be the residue of $t_i$ modulo $2$. We define
\[
a(\Delta) = \sum_{i=1}^q r_i, \quad b(\Delta) = \sum_{i=1}^q s_i + \sum_{i=1}^q \frac{t_i - r_i}{2}, \quad \text{and} \quad c(\Delta) = \frac{|\Delta| - a(\Delta) - 2b(\Delta)}{3}.
\]
Let $k(\Delta) = a(\Delta) + b(\Delta) + c(\Delta)$. We then define the piecewise function
\[
d(\Delta, t)=
\begin{cases}
k(\Delta) - t + 2, & \text{if } 1 \le t \le a(\Delta), \\[6pt]
k(\Delta)+1- \left\lfloor\frac{t + a(\Delta)-1}{2} \right\rfloor, 
& \text{if } a(\Delta) + 1 \le t \le a(\Delta) + 2b(\Delta), \\[8pt]
k(\Delta)+1- \left\lfloor \frac{t + 2a(\Delta) + b(\Delta) - 1}{3} \right\rfloor, 
& \text{if } a(\Delta) + 2b(\Delta) + 1 \le t \le |\Delta|, \\[8pt]
1, & \text{if } t \ge |\Delta| + 1.
\end{cases}
\]

The main result of this paper is stated as follows.

\begin{Theorem}\label{maintheorem}
Let $P_{\bf w}^{n+1}$ be an increasing weighted path. Then, for any $t \ge 1$, 
\[
\depth(S/I(P_{\bf w}^{n+1})^t)=d(\Delta, t).
\]
\end{Theorem}

\begin{Example}
Let 
\[
I = (x_1x_2,x_2x_3,x_3^2x_4^2,x_4^2x_5^2,x_5^2x_6^2,
x_6^3x_7^3,x_7^3x_8^3,x_8^3x_9^3,
x_9^4x_{10}^4,x_{10}^4x_{11}^4,
x_{11}^5x_{12}^5).
\]
Then 
\[
\Delta =\delta_1\cup \delta_2 \cup \delta_3\cup \delta_4,
\]
where $\delta_1= \{1\}, \delta_2=\{3,4\}, \delta_3=\{6,7\}, \delta_4=\{9\}$. Note that $\min(\delta_{i+1})-\max(\delta_i)= 2$ for $i=1,2,3$. Thus, $\Delta$ consists of one extended block containing two type 1 blocks and two type 2 blocks. Therefore, $a(\Delta) = 0$, $b(\Delta) = 3$, and $c(\Delta) = 0$. Consequently, the depth sequence $\depth(S/I^t)$ for $t = 1, \ldots, 7$ is 
\[
\{4,4,3,3,2,2,1\}.
\]
\end{Example}

    The paper is organized as follows: Section \ref{sec:prelim} introduces the basic facts that are used throughout the paper. In Section \ref{sec:results}, we prove Theorem \ref{maintheorem}.

\section{Preliminaries}
\label{sec:prelim}
	
In this section, we provide some basic facts that will be used throughout this paper. Detailed information can be found in \cite{BH} and \cite{HH}.
For a finitely generated graded $S$-module $L$, the depth of $L$ is defined as follows:
\[
\depth(L) = \min\{i \mid H_{\mm}^i(L) \ne 0\},
\]
where $H_{\mm}^{i}(L)$ denotes the $i$-th local cohomology module of $L$ with respect to $\mm$.

Let $G$ be a simple graph. For any subset $A$ of $V(G)$, $G[A]$ denotes the \emph{induced subgraph} of $G$ on the set $A$; that is, for any $u,v \in A$, $\{u,v\} \in E(G[A])$ if and only if $\{u,v\} \in E(G)$. For a vertex $v \in V(G)$, its \emph{neighborhood} and \emph{closed neighborhood} are defined by $N_G(v) := \{u \in V(G) \mid \{u,v\} \in E(G)\}$ and $N_G[v] := N_G(v) \cup \{v\}$, respectively. By abuse of notation, we also write $uv$ for the edge $\{u,v\}$ of $G$.

Let $G_{\mathbf{w}}$ be an edge-weighted graph. Any concept valid for $G$ extends 
naturally to $G_{\mathbf{w}}$. For example, for a subset $A \subseteq V(G)$, 
the \emph{induced subgraph} $G_{\mathbf{w}}[A]$ is the graph with underlying 
graph $G[A]$, where the weight of each edge in $G_{\mathbf{w}}[A]$ is 
identical to its weight in $G_{\mathbf{w}}$.

	\begin{Lemma}{\em (\cite[Lemma 2.2]{HT})}
			\label{sum1}
			Let $S_{1}=\KK[x_{1},\dots,x_{m}]$ and $S_{2}=\KK[x_{m+1},\dots,x_{n}]$ be two polynomial rings  over a field $\KK$. Let $S=S_1\otimes_\KK S_2$ and  $I\subset S_{1}$,
			$J\subset S_{2}$ be two nonzero homogeneous  ideals.  Then 
\[ 
\depth\bigl(S/(I+J)\bigr)=\depth(S_1/I)+\depth(S_2/J).
\]
		\end{Lemma}

	\begin{Lemma}  {\em (\cite[Lemma 2.1]{HT})}
		\label{exact}
		Let $0\longrightarrow M\longrightarrow N\longrightarrow P\longrightarrow 0$ be an exact
		sequence of finitely generated graded $S$-modules. Then 
\begin{itemize}
			\item[(1)] $\depth(M)\geq \min\{\depth(N), \depth(P)+1\}$, the equality holds if $\depth(N) \neq \depth(P)$;
			\item[(2)] $\depth(N)\geq \min\{\depth(M), \depth(P)\}$, the equality holds if $\depth(P) \neq \depth(M)-1$.
		\end{itemize}
	\end{Lemma}
		
\begin{Lemma}{\em (\cite[Corollary 1.3]{R})}\label{radup}
Let $I\subset S$ be a monomial ideal and $f \notin I$ be a monomial. Then 
\[
\depth( S / I) \leq \depth\bigl(S /(I: f)\bigr).
\]
\end{Lemma}

\begin{Lemma}\label{lem_consequence_Depth_Lemma}Let $I$ be   a homogeneous ideal and $f$ be a nonzero homogeneous form of $S$. Then 
$$\depth (S/I) \ge \min \{\depth (S/(I:f)), \depth (S/(I,f))\}.$$        
\end{Lemma} 
\begin{proof}
    Applying Lemma \ref{exact} to the short exact sequence 
    $$0 \to S/(I:f) \to S/I \to S/(I,f) \to 0,$$
    we obtain the conclusion.
\end{proof}

\begin{Lemma}\label{lem_Depth_intersection}Let $I,J$ be nonzero homogeneous ideals of $S$. Then 
$$\depth (S/(I \cap J)) \ge \min \{\depth (S/I), \depth (S/J), \depth (S/(I+J)) + 1\}.$$        
\end{Lemma} 
\begin{proof}
    Applying Lemma \ref{exact} to the short exact sequence 
    $$0 \to S/(I \cap J) \to S/I \oplus S/J \to S/(I + J) \to 0,$$
    we obtain the conclusion.
\end{proof}

Recall that a weighted graph $G_{\bf w}$ is said to be \emph{nontrivially weighted} if there is at least one edge with weight greater than $1$; otherwise, it is said to be \emph{trivially weighted}. An edge $e \in E(G_{\bf w})$ is \emph{nontrivially weighted} if ${\bf w}(e) \ge 2$; otherwise, we say that $e$ has \emph{trivial weight}.

In this paper, we extend a lemma from \cite{LVZ}.

\begin{Lemma}\label{lem_colon_leaf} 
	Let $G_{\bf w}$ be an edge-weighted graph, and let $e = x_{n-1}x_n$ be a leaf edge of $G$, that is, $x_n$ has a unique neighbor, namely $x_{n-1}$. If ${\bf w}(e)\le {\bf w}(x_{n-1}x_j)$ for all $x_j\in N_G(x_{n-1})$, then for all $t \ge 2$, we have
\[
	\bigl(I(G_{\bf w})^t : (x_{n-1}x_n)^{{\bf w}(e)}\bigr)= I(G_{\bf w})^{t-1}.
\]    
\end{Lemma}
\begin{proof}
	Clearly, the left-hand side contains the right-hand side. It therefore suffices to prove that if $f$ is a minimal generator of $I(G_{\bf w})^t$, then 
	$$g = f / \gcd(f, (x_{n-1}x_n)^{{\bf w}(e)}) \in I(G_{\bf w})^{t-1}.$$
	To this end, we can write $f = f_1 \cdots f_t$, where each $f_i$ is a minimal generator of $I(G_{\bf w})$. If $x_n$ divides $f_j$ for some $j$, then $f_j = (x_{n-1}x_n)^{{ \bf w}(e)}$ since $x_n$ is a leaf of $G$. Hence,
\[
	g = f / \gcd(f, (x_{n-1}x_n)^{{\bf w}(e)}) = f_1 \cdots f_{j-1} f_{j+1} \cdots f_t \in I(G_{\bf w})^{t-1}.
\]
	
	Now assume that $x_n$ does not divide $f_j$ for any $1\le j\le t$. Since ${\bf w}(e)\le {\bf w}(x_{n-1}x_j)$ for all $x_j\in N_G(x_{n-1})$, $\gcd(f,(x_{n-1}x_n)^{{\bf w}(e)})$ is either $1$ or $x_{n-1}^{{\bf w}(e)}$. In the first case, $g = f \in I(G_{\bf w})^t \subseteq I(G_{\bf w})^{t-1}$ (as $t \ge 2$). In the second case, there exists some $j$ such that $f_j$ is divisible by $x_{n-1}$; hence, 
	$$g \text{ is divisible by } f_1 \cdots f_{j-1} f_{j+1} \cdots f_t \in I(G_{\bf w})^{t-1}.$$
	The conclusion follows.
\end{proof}

We use the following notation throughout the remainder of this paper. Let $m$ and $n$ be integers with $m \le n$. By convention, the notation $[m,n]$ denotes the set $\{m, m+1, \ldots, n\}$. In particular, if $m = n$, then $[m,n] = \{m\}$, and $[n]$ is defined as $[1,n]$. Let $P_{\bf w}^{n+1}$ be a nontrivially weighted path, and let $i$ and $j$ be integers in $[n+1]$. If $i \le j$, then $P_{\bf w}^{[i,j]}$ denotes the induced subpath of $P_{\bf w}^{n+1}$ on the vertex set $\{x_i, x_{i+1}, \ldots, x_j\}$; otherwise, $P_{\bf w}^{[i,j]}$ is defined to be the empty graph.
For a subset $A \subseteq [n]$, we define 
\[
\max(A) = \max\{a \mid a \in A\}
\quad \text{and} \quad
\min(A) = \min\{a \mid a \in A\}.
\]

\begin{Definition}\em
	Let $n$ be a positive integer and let $T \subseteq [n]$ be a nonempty subset. 	For integers $p, q$ with $1 \le p \le q \le n$, the interval $[p,q]$ is called a \emph{maximal block} of $T$ if the following three conditions hold:
	\begin{enumerate}
		\item $[p,q] \subseteq T$,
		\item $p - 1 \notin T$,
		\item $q + 1 \notin T$.
	\end{enumerate}
\end{Definition}

By the above definition, $T$ can be uniquely expressed as a disjoint union of maximal blocks:
\[
T = \delta_1 \sqcup \delta_2 \sqcup \cdots \sqcup \delta_s.
\]
This decomposition is referred to as the \emph{maximal block decomposition} of $T$.

\begin{Definition}
The type of a block $\delta = [p,q]$ is defined as the residue of $|\delta|$ modulo $3$, which takes values in $\{0,1,2\}$.
\end{Definition}


\begin{Definition}
 Let $\delta_1 = [p_1,q_1]$ and $\delta_2 = [p_2,q_2]$ be two blocks with $q_1 < p_2$. We say that $\delta_1$ and $\delta_2$ are \emph{gluable} if $|\delta_1| \not\equiv 0 \pmod{3}$, $|\delta_2| \not\equiv 0 \pmod{3}$, and $p_2 - q_1 = 2$. Let $T \subseteq [n]$ be a nonempty subset with maximal block decomposition $T = \delta_1 \sqcup \cdots \sqcup \delta_s$. A union $M = \delta_j \cup \cdots \cup \delta_k$ is called a \emph{maximal extended block} of $T$ if the following conditions are satisfied:
	\begin{enumerate}
		\item $\delta_\ell$ is of type $1$ or type $2$ for all $\ell = j, \ldots,k$,
		\item $\delta_\ell$ and $\delta_{\ell+1}$ are gluable for all $\ell = j, \ldots, k-1$,
		\item either $j=1$ or $\delta_{j-1}$ is not gluable to $\delta_j$,
		\item either $k=s$ or $\delta_k$ is not gluable to $\delta_{k+1}$.
	\end{enumerate}
\end{Definition}

We recall the following definition.
\begin{Definition}
Let $\Delta \subseteq [n]$ be a subset of indices, and suppose $\Delta$ admits the decomposition
\[
    \Delta = \Delta_1 \cup \cdots \cup \Delta_q,
\]
where each $\Delta_i$ is either a block of type $0$ or a maximal extended block formed by gluing together type $1$ and type $2$ blocks.

For $i = 1,\ldots,q$, let $t_i$ and $s_i$ denote the numbers of type $1$ and type $2$ blocks in $\Delta_i$, respectively, and let $r_i \in \{0,1\}$ be the residue of $t_i$ modulo $2$. We define
\[
a(\Delta) = \sum_{i=1}^q r_i, \quad b(\Delta) = \sum_{i=1}^q s_i + \sum_{i=1}^q \frac{t_i - r_i}{2}, \quad \text{and} \quad c(\Delta) = \frac{|\Delta| - a(\Delta) - 2b(\Delta)}{3}.
\]
Let $k(\Delta) = a(\Delta) + b(\Delta) + c(\Delta)$. We define the piecewise function $d(\Delta, t)$ as follows:
\[
d(\Delta, t) = 
\begin{cases} 
    k(\Delta) - t + 2, & \text{if } 1 \le t \le a(\Delta), \\[6pt]
    k(\Delta) + 1 - \left\lfloor \frac{t + a(\Delta) - 1}{2} \right\rfloor, & \text{if } a(\Delta) + 1 \le t \le a(\Delta) + 2b(\Delta), \\[8pt]
    k(\Delta) + 1 - \left\lfloor \frac{t + 2a(\Delta) + b(\Delta) - 1}{3} \right\rfloor, & \text{if } a(\Delta) + 2b(\Delta) + 1 \le t \le |\Delta|, \\[8pt]
    1, & \text{if } t \ge |\Delta| + 1.
\end{cases}
\]
\end{Definition}

\begin{Lemma}\label{D1}
Let $\Delta \subseteq [n]$ be a nonempty set of indices and let $\Gamma_1 = \Delta\setminus\{\min(\Delta)\}$. Then for all $t \ge 1$,
\[
d(\Delta,t) \le \min\{ d(\Gamma_1,t-1),\; d(\Gamma_1,t)+1 \}.
\]
\end{Lemma}

\begin{proof}
Let $\Delta = \delta_1 \cup \cdots \cup \delta_p$ be the decomposition of $\Delta$ into maximal blocks, where $\delta_1$ is the first maximal block.
Then the blocks of $\Gamma_1$ are exactly $\delta_1\setminus\{\min(\Delta)\},\delta_2,\ldots,\delta_p$ and $|\Gamma_1|=|\Delta|-1$. We distinguish three cases.

Case 1. Suppose that $\delta_1$ is of type $2$. If $\delta_2$ is not gluable to $\delta_1$, then direct computation yields
$a(\Delta) = a(\Gamma_1)-1$, $b(\Delta) = b(\Gamma_1)+1$, $c(\Delta) = c(\Gamma_1)$, and the desired inequality holds.

If $\delta_2$ is gluable to $\delta_1$, then
$a(\Delta) = a(\Gamma_1) - 1$, $b(\Delta) = b(\Gamma_1) + 1$, $c(\Delta) = c(\Gamma_1)$ if $t_1 \equiv 0 \pmod 2$, and
$a(\Delta) = a(\Gamma_1) + 1$, $b(\Delta) = b(\Gamma_1)$, $c(\Delta) = c(\Gamma_1)$ if $t_1 \equiv 1 \pmod 2$.
The conclusion follows similarly.

Case 2. Assume that $\delta_1$ is of type $1$. If $\delta_2$ is not gluable to $\delta_1$, then
$a(\Delta) = a(\Gamma_1)+1$, $b(\Delta) = b(\Gamma_1)$, $c(\Delta) = c(\Gamma_1)$, and the claim holds.

If $\delta_2$ is gluable to $\delta_1$, then
$a(\Delta) = a(\Gamma_1) - 1$, $b(\Delta) = b(\Gamma_1) + 1$, $c(\Delta) = c(\Gamma_1)$ if $t_1 \equiv 0 \pmod 2$, and
$a(\Delta) = a(\Gamma_1) + 1$, $b(\Delta) = b(\Gamma_1)$, $c(\Delta) = c(\Gamma_1)$ if $t_1 \equiv 1 \pmod 2$.
The result is again straightforward.

Case 3. Assume that $\delta_1$ is of type $0$. Then $\Delta_1=\delta_1$ and $t_1=0$. Thus,
$a(\Delta) = a(\Gamma_1)$, $b(\Delta) = b(\Gamma_1)-1$, $c(\Delta) = c(\Gamma_1)+1$.
The inequality follows immediately.
\end{proof}

\begin{Lemma}\label{D2}
Let $\Delta \subseteq [n]$ be a set of indices with $1\in\Delta$, and let $\Gamma_2 = \Delta \cap [3,n]$. Then for all $t \ge 1$, we have
\[
d(\Delta,t) \le d(\Gamma_2,t) + 1.
\]
\end{Lemma}

\begin{proof}
Let $\Delta = \delta_1 \cup \cdots \cup \delta_p$ be the decomposition of $\Delta$ into maximal blocks, where $\delta_1$ denotes the first maximal block.
We proceed by considering the following cases.

Case 1. Suppose $|\delta_1| = 1$. Then $2\notin \Delta$, and $\Gamma_2=\Delta\setminus\{1\}=\Delta\setminus\{\min(\Delta)\}$. By Lemma \ref{D1}, the desired inequality holds immediately.

Case 2. Suppose $|\delta_1|\ge 2$. Then the blocks of $\Gamma_2$ are exactly $\delta_1\setminus\{1,2\},\delta_2,\ldots,\delta_p$, and $|\Gamma_2|=|\Delta|-2$. We further distinguish three subcases:

(i) If $\delta_1$ is of type $2$, then $a(\Delta) = a(\Gamma_2)$, $b(\Delta) = b(\Gamma_2)+1$, and $c(\Delta) = c(\Gamma_2)$, so the claim follows.

(ii) Assume $\delta_1$ is of type $1$. If $\delta_2$ is not gluable to $\delta_1$, then $a(\Delta) = a(\Gamma_2)+1$, $b(\Delta) = b(\Gamma_2)-1$, and $c(\Delta) = c(\Gamma_2)+1$, and the inequality holds.

If $\delta_2$ is gluable to $\delta_1$, then
\[
\begin{cases}
a(\Delta) = a(\Gamma_2) - 1,\; b(\Delta) = b(\Gamma_2),\; c(\Delta) = c(\Gamma_2)+1 & \text{if } t_1 \equiv 0 \pmod 2, \\
a(\Delta) = a(\Gamma_2) + 1,\; b(\Delta) = b(\Gamma_2)-1,\; c(\Delta) = c(\Gamma_2)+1 & \text{if } t_1 \equiv 1 \pmod 2.
\end{cases}
\]
The conclusion is clear in both subcases.

(iii) Assume $\delta_1$ is of type $0$. Then $\Delta_1=\delta_1$ and $t_1=0$. 
If $\min(\delta_2)-\max(\delta_1)\ge 3$, then $a(\Delta) = a(\Gamma_2)-1$, $b(\Delta) = b(\Gamma_2)$, and $c(\Delta) = c(\Gamma_2)+1$, so the claim is verified.

If $\min(\delta_2)-\max(\delta_1)=2$, then
\[
\begin{cases}
a(\Delta) = a(\Gamma_2) - 1,\; b(\Delta) = b(\Gamma_2),\; c(\Delta) = c(\Gamma_2)+1 & \text{if } t_2 \equiv 0 \pmod 2, \\
a(\Delta) = a(\Gamma_2) + 1,\; b(\Delta) = b(\Gamma_2)-1,\; c(\Delta) = c(\Gamma_2)+1 & \text{if } t_2 \equiv 1 \pmod 2.
\end{cases}
\]
The inequality holds in this case as well.
\end{proof}

\begin{Lemma}\label{D3}
Let $\Delta \subseteq [n]$ be a set of indices with $1\in\Delta$, and let $\Gamma_3 = \Delta \cap [4,n]$. Then for all $t \ge 1$,
\[
d(\Delta,t) \le d(\Gamma_3,t) + 1.
\]
\end{Lemma}

\begin{proof}
Let $\Delta = \delta_1 \cup \cdots \cup \delta_p$ be the decomposition of $\Delta$ into maximal blocks, where $\delta_1$ denotes the first maximal block.
We consider the following cases.

Case 1. Suppose $|\delta_1| = 1$.
If $2,3\notin \Delta$, then $\Gamma_3=\Delta\setminus\{1\}=\Delta\setminus\{\min(\Delta)\}$. By Lemma~\ref{D1}, the desired inequality holds.

If $2\notin \Delta$ and $3\in\Delta$, then
the blocks of $\Gamma_3$ are exactly $\delta_2\setminus\{3\},\delta_3,\ldots,\delta_p$ and $|\Gamma_3|=|\Delta|-2$.
We distinguish three subcases.

(i) If $\delta_2$ is of type $2$, then $\delta_2\setminus\{3\}$ is of type $1$.
It follows that $a(\Delta) = a(\Gamma_3)$, $b(\Delta) = b(\Gamma_3)+1$, and $c(\Delta) = c(\Gamma_3)$. The claim holds.

(ii) If $\delta_2$ is of type $1$, then $a(\Delta) = a(\Gamma_3)$, $b(\Delta) = b(\Gamma_3)+1$, and $c(\Delta) = c(\Gamma_3)$. The result is straightforward.

(iii) If $\delta_2$ is of type $0$, then $a(\Delta) = a(\Gamma_3)+1$, $b(\Delta) = b(\Gamma_3)-1$, and $c(\Delta) = c(\Gamma_3)+1$. The conclusion follows.

Case 2. Suppose $|\delta_1| = 2$. Then $3\notin \Delta$ and $\Gamma_3=\Delta\setminus\{1,2\}$. By Lemma~\ref{D2}, the inequality holds immediately.

Case 3. Suppose $|\delta_1| \ge 3$.
Then the blocks of $\Gamma_3$ are exactly $\delta_1\setminus\{1,2,3\},\delta_2,\ldots,\delta_p$ and $|\Gamma_3|=|\Delta|-3$.
Since $\delta_1$ and $\delta_1\setminus\{1,2,3\}$ are of the same type, we have
$a(\Delta) = a(\Gamma_3)$, $b(\Delta) = b(\Gamma_3)$, and $c(\Delta) = c(\Gamma_3)+1$.
The desired estimate follows at once.
\end{proof}

\section{Main results}
\label{sec:results}

In this section, we provide the exact formula for the depth of an increasing weighted path $P_{\bf w}^{n+1}$.  

If $n \le 2$, then $P_{\bf w}^{n+1}$ is a star graph. By \cite[Theorems 3.1 and 3.3]{ZDCL}, we have $\depth(S/I(P_{\bf w}^{n+1})^t) = 1$ for all $t \ge 1$.

Thus, we assume $n \ge 3$. For simplicity, we set ${\bf w}_i = {\bf w}(e_i)$ for all $i \in [n]$. In this case, if $P_{\bf w}^{n+1}$ is a strictly increasing weighted path (i.e., ${\bf w}_i < {\bf w}_{i+1}$ for all $i \in [n-2]$), then by \cite[Proposition 3.8]{HLTZ}, $\depth(S/I(P_{\bf w}^{n+1})^t) = 1$ for all $t \ge 1$. From this point onward, we therefore assume that $P_{\bf w}^{n+1}$ is an increasing but not strictly increasing weighted path. That is, there exists some $i \in [n-2]$ such that ${\bf w}_i = {\bf w}_{i+1}$. Let $\Delta = \{i \in [n-2] \mid {\bf w}_i = {\bf w}_{i+1}\}$; then $\Delta \ne \emptyset$.

 In the following, unless otherwise specified, we will assume that $f_i=(x_ix_{i+1})^{{\bf w}_i}$ for all $i\in [n]$.

By \cite[Theorem 3.7]{HLTZ}, $\depth(S/I(P_{\bf w}^{n+1})^t) = 1$ for all $t \ge |\Delta| + 1$. Therefore, it suffices to determine $\depth(S/I(P_{\bf w}^{n+1})^t)$ for all $t \in [1, |\Delta|]$.

First, we establish the lower bound. We need the following lemma.

\begin{Lemma}\label{lem:colon_x2} Assume that $n \ge 2$. Let $P_{\bf w}^{n+1}$ be a weighted path such that ${\bf w}_1={\bf w}_2$. Denote by $I = I(P_{\bf w}^{n+1})$ and $J = I(P_{\bf w}^{[4,n+1]})$. Then for all $t \ge 1$ we have 
\[
(I^t : x_2^{{\bf w}_1}) = (x_1^{{\bf w}_1}, x_3^{{\bf w}_1})I ^{t-1} + J^t.
\]
\end{Lemma}
\begin{proof} Clearly the left hand side contains the right hand side. Now assume that 	$g =(\prod\limits_{j=1}^{t} f_{i_j}) / \gcd(\prod\limits_{j=1}^{t} f_{i_j},x_2^{{\bf w}_1})$ is a minimal generator of 
	$(I^t : x_2^{{\bf w}_1})$.
   We may assume that $i_1 \le \cdots \le i_t$. 
  If $i_1 >2$, then $g =\prod\limits_{j=1}^{t} f_{i_j}\in J^t + 
 x_3^{{\bf w}_1}I^{t-1}.$

If $i_1 = 1$, then $f_{i_1} = x_1^{{\bf w}_1} x_2^{{\bf w}_1}$ hence $g = x_1^{{\bf w}_1}\prod\limits_{j=2}^{t} f_{i_j} \in x_1^{{\bf w}_1}I^{t-1}$.

If $i_1 = 2$, then $f_{i_1} =	x_2^{{\bf w}_1} x_3^{{\bf w}_1}
$ hence, $g = x_3^{{\bf w}_1}\prod\limits_{j=2}^{t} f_{i_j}  \in x_3^{{\bf w}_1}I^{t-1}$. The conclusion follows.    
\end{proof}

\begin{Lemma}\label{lowerbounds}
$\depth (S/I(P_{\bf w}^{n+1})^t) \ge d(\Delta,t)$ for all $t \ge 1$.
\end{Lemma}

\begin{proof}
Let $I = I(P_{\bf w}^{n+1})$. We proceed by induction on $n$ and   $t$. The base case $n \le 2$ is trivial. We now assume $n \ge 3$ and consider two cases:

(I) Assume ${\bf w}_1 < {\bf w}_2$. First, let $t = 1$. Assume that  $L=I(P_{\bf w}^{[2,n+1]})$. Then 
$I =(L,x_1^{{\bf w}_1}) \cap (L,x_2^{{\bf w}_1})$. By Lemma \ref{lem_Depth_intersection}, we obtain
\[
\depth\bigl(\frac{S}{I}\bigr) \ge \min \Bigl\{\depth\Bigr(\frac{S}{(L,x_1^{{\bf w}_1})}\Bigr),\,\depth\Bigl(\frac{S}{(L,x_2^{{\bf w}_1})}\Bigr),\,
\depth\Bigl(\frac{S}{(L,x_1^{{\bf w}_1},x_2^{{\bf w}_1})}\Bigr)+1\Bigr\}.
\]
Note that the set of indices $\Delta$ is the same for $I$ and $L$. Hence, by induction, $\depth (S/(L,x_1^{{\bf w}_1})) \ge d(\Delta,t)$. The conclusion follows from the induction hypothesis on $n$ and Lemmas \ref{sum1}, \ref{D1}, and \ref{D2}.

Next, assume that  $t > 1$. Let $f = (x_1x_2)^{{\bf w}_1}$, then $(I^t:f) = I^{t-1}$ and $(I^t,f) = (I^t,x_1^{{\bf w}_1}) \cap (I^t ,x_2^{{\bf w}_1})$. By  similar arguments to the preceding paragraph, we deduce that $\depth (S/(I^t,f)) \ge d(\Delta,t)$. By Lemma \ref{lem_consequence_Depth_Lemma}, the induction hypothesis on $t$, we deduce
\[
\depth (S/I^t) \ge \min\left\{\depth \bigl(S/(I^t : f)\bigr), \depth\bigl (S/(I^t,f)\bigr)\right\} \ge d(\Delta,t).
\]

(II) Assume ${\bf w}_1 = {\bf w}_2$. For $t = 1$, we have $(I:x_2^{{\bf w}_1})=I(P_{\bf w}^{[4,n+1]})+(x_1^{{\bf w}_1},x_3^{{\bf w}_2})$ and $(I,x_2^{{\bf w}_1})=I(P_{\bf w}^{[3,n+1]})+(x_2^{{\bf w}_1})$. By the induction hypothesis on $n$ and Lemmas \ref{sum1}, \ref{D2}, and \ref{D3}, we have 
\[
\depth (S/(I:x_2^{{\bf w}_1}))=d(\Delta_1,1)+1\ge d(\Delta,1)
\]
and
\[
\depth (S/(I,x_2^{{\bf w}_1}))=d(\Delta_2,1)+1\ge d(\Delta,1),
\]
where $\Delta_1=\Delta\cap[4,n-2]$ and $\Delta_2=\Delta\cap[3,n-2]$. Applying Lemma \ref{lem_consequence_Depth_Lemma}, we obtain $\depth (S/I)\ge d(\Delta,1)$.

Now assume that  $t > 1$. By Lemma \ref{lem_consequence_Depth_Lemma}, we have 
\[
\depth (S/I^t) \ge \min \{\depth (S/(I^t,x_2^{{\bf w}_1})), \depth (S/(I^t:x_2^{{\bf w}_1}))\}.
\]
By induction and Lemma \ref{D1}, we have $\depth (S/(I^t,x_2^{{\bf w}_1})) \ge d(\Delta,t)$. Now, by Lemma \ref{lem:colon_x2}, we have 
$K = (I^t : x_2^{{\bf w}_1}) =(x_1^{{\bf w}_1}, x_3^{{\bf w}_1}) I^{t-1} + J^t$, where $J=I(P_{\bf w}^{[4,n+1]})$.

By Lemma \ref{lem_consequence_Depth_Lemma}, we have $\depth (S/K) \ge \min \{\depth (S/(K,x_1^{{\bf w}_1})), \depth (S/(K : x_1^{{\bf w}_1}))\}$. Now,
 $(K : x_1^{{\bf w}_1}) = I^{t-1}$ and $(K,x_1^{{\bf w}_1}) = (x_1^{{\bf w}_1})+x_3^{{\bf w}_1} L^{t-1} +J^{t}$. 
By induction on $t$, we have $\depth (S/(K : x_1^{{\bf w}_1})) \ge d(\Delta,t)$. 
 By induction  on $n$,  Lemmas \ref{exact},  \ref{D1} and \ref{D3}, we can deduce that 
\begin{align*}
\depth (S/(K,x_1^{{\bf w}_1})) &\ge \min \{\depth (S/ (K,x_1^{{\bf w}_1},x_3^{{\bf w}_1})), \depth (S/ ((K,x_1^{{\bf w}_1}) : x_3^{{\bf w}_1}))\}\\
& \ge d(\Delta,t).
\end{align*}
 The conclusion follows.
\end{proof}

\begin{Theorem}\label{monoincrease2}
$\depth \bigl(S/I(P_{\bf w}^{n+1})\bigr) = k(\Delta) + 1$, where $k(\Delta) = a(\Delta) + b(\Delta) + c(\Delta)$.
\end{Theorem}

\begin{proof}
Let $I = I(P_{\bf w}^{n+1})$. By Lemma \ref{lowerbounds}, it suffices to show that $\depth (S/I) \le k(\Delta) + 1$. We proceed by induction on $n$. The base case $n \le 2$ is trivial, so we assume $n \ge 3$.

Let
\[
f = x_{n+1}^{{\bf w}_n - 1}\biggl(\prod_{k=m+2}^{n}x_k^{{\bf w}_{k-1}}\biggr),
\]
where $m = \max(\Delta)$. Then
\[
(I : f) = I(P_{\bf w}^{[1,m]}) + \bigl(x_{m+1}^{{\bf w}_{m+1}}, x_{m+2}^{{\bf w}_{m+2} - {\bf w}_{m+1}}, x_{m+3}^{{\bf w}_{m+3} - {\bf w}_{m+2}}, \ldots,
 x_{n-1}^{{\bf w}_{n-1} - {\bf w}_{n-2}}, x_{n}^{{\bf w}_{n} - {\bf w}_{n-1}}x_{n+1}\bigr).
\]
By Lemma \ref{radup}, it suffices to prove that $\depth(S/(I:f)) \le k(\Delta) + 1$. We consider two cases:

(i) If $m \in [1,3]$, then $k(\Delta) = 1$. By Lemma \ref{sum1}, we have $\depth(S/(I:f)) = 2 = k(\Delta) + 1$.

(ii) If $m \ge 4$, let $R = \KK[x_1, \ldots, x_m]$ and $\Delta' = \{i \in [m-3] \mid {\bf w}_i = {\bf w}_{i+1}\}$. By the induction hypothesis on $n$ and Lemma \ref{sum1}, we have $\depth(R/I(P_{\bf w}^{[1,m]})) \le k(\Delta') + 1$. Moreover, Lemma \ref{D3} gives $k(\Delta') \le k(\Delta) - 1$. Combining these results with Lemma \ref{sum1}, we deduce
\[
\depth\bigl(S/(I:f)\bigr) = \depth\bigl(R/I(P_{\bf w}^{[1,m]})\bigr) + 1 \le k(\Delta) + 1.
\]

This completes the proof.
\end{proof}

In the following, we consider the case where $t\ge 2$. If $|\Delta|=1$, then,   by \cite[Theorem 3.7]{HLTZ},
 $\depth(S/I(P_{\bf w}^{n+1})^t) = 1$ for all $t \ge 2$. Hence, we always assume that $|\Delta|\ge 2$.
We will prove that $\depth(S/I(P_{\bf w}^{n+1})^t) \le d(\Delta, t)$. To this end, we construct a monomial $f$ such that $\depth (S/I^t) \le \depth \bigl(S/(I^t:f)\bigr) = d(\Delta,t)$. The construction of $f$ relies on the following definition.

Let $\Delta = \delta_1 \cup \cdots \cup \delta_q$ be the decomposition of $\Delta$ into maximal blocks. We define a partition
\[
\Delta = A(\Delta) \sqcup B(\Delta) \sqcup C(\Delta).
\]
It suffices to recursively define $A(\Delta)$ and $B(\Delta)$; the set $C(\Delta)$ is then given by $C(\Delta) = \Delta \setminus (A(\Delta) \cup B(\Delta))$.

For convenience, we write each block as $\delta_j = [m_j,n_j]$ for $j = 1,\ldots,q$.

\medskip
\noindent
\textbf{Base case.} 
If $\Delta$ consists of a single block $\delta_1$, then
\[
A(\Delta)=
\begin{cases}
\{n_1\}, & \text{if $\delta_1$ is of type $1$},\\
\emptyset, & \text{otherwise},
\end{cases}
\;\;
B(\Delta)=
\begin{cases}
\{n_1-1,n_1\}, & \text{if $\delta_1$ is of type $2$},\\
\emptyset, & \text{otherwise}.
\end{cases}
\]

\medskip
\noindent
\textbf{Recursive step.}
Let $\Delta' = \delta_2 \cup \cdots \cup \delta_{q}$.

\begin{itemize}
\item If $\delta_1$ is of type $0$ or type $2$, then
\[
A(\Delta) = A(\Delta'), \qquad 
B(\Delta) = B(\Delta') \cup B(\delta_1).
\]

\item If $\delta_1$ is of type $1$ and not gluable to $\delta_{2}$, then
\[
A(\Delta) = A(\Delta') \cup A(\delta_1), \qquad 
B(\Delta) = B(\Delta').
\]

\item If $\delta_1$ is of type $1$ and gluable to $\delta_{2}$, set
\[
\Gamma = \delta_2' \cup \cdots \cup \delta_{q-1} \cup \delta_{q},
\quad \text{where } 
\delta_{2}' = \delta_{2} \setminus \{m_{2}\}.
\]
We then define
\[
A(\Delta) = A(\Gamma), \qquad 
B(\Delta) = B(\Gamma) \cup \{n_{1},m_2\}.
\]
\end{itemize}
    
  It is easy to see that the cardinalities of $A(\Delta)$, $B(\Delta)$, and $C(\Delta)$ are related to the previously defined $a(\Delta)$, $b(\Delta)$, and  $c(\Delta)$ as follows:  
  $|A(\Delta)|=a(\Delta)$, $|B(\Delta)|=2b(\Delta)$ and $|C(\Delta)|=3c(\Delta)$.

\medskip
We illustrate the construction with the following example.

\begin{Example}
Let 
\[
\Delta = \{1\} \cup \{3,4\} \cup \{6\} \cup \{8,9\} 
\cup \{11,12,13\} \cup \{15\}.
\]

The first block $\{1\}$ is of type $1$ and gluable to the subsequent block $\{3,4\}$. Thus, $\{1,3\} \subseteq B(\Delta)$, and we proceed with the remaining blocks $\{4\} \cup \{6\} \cup \{8,9\} \cup \{11,12,13\} \cup \{15\}$.

The block $\{4\}$ is of type $1$ and gluable to the subsequent block $\{6\}$, so $\{4,6\} \subseteq B(\Delta)$.

The block $\{8,9\}$ is of type $2$, hence $\{8,9\} \subseteq B(\Delta)$.

Next, the block $\{11,12,13\}$ is of type $0$, so $\{11,12,13\}$ contributes to $C(\Delta)$.

Finally, the remaining singleton block $\{15\}$ contributes to $A(\Delta)$.

We therefore have
\[
A(\Delta)=\{15\}, \qquad 
B(\Delta)=\{1,3,4,6,8,9\}, \qquad 
C(\Delta)=\{11,12,13\}.
\]
\end{Example}

Now, we label the elements of $\Delta$ by first listing the elements of $A(\Delta)$, followed by those of $B(\Delta)$, and finally those of $C(\Delta)$. Each set is arranged in descending order. We denote this labeling by $\mu_1, \ldots, \mu_{|\Delta|}$.  For the example above, we have
\[
\{\mu_1,\ldots,\mu_{10}\}=
\{15,9,8,6,4,3,1,13,12,11\}.
\]

\begin{Remark}\label{rem1}
Let $t \in [2, |\Delta|+1]$, and let $g_t$ denote the product of the monomials corresponding to the weighted edges $e_{\mu_j + 1}$ for all $j = 1, \ldots, t-1$. By Lemma \ref{radup}, it suffices to show that $\depth\bigl(S/(I(P_{\bf w}^{n+1})^t : g_t)\bigr) \le d(\Delta, t)$. However, the colon ideal $\bigl(I(P_{\bf w}^{n+1})^t : g_t\bigr)$ is often too complex to directly verify the depth formula.

For example, let ${\bf w} = (1,1,2,2,2,2,4,4,6,6,6)$. Here, $\Delta = \{1, 3, 4, 5, 7, 9\}$. If we set $g_2 =f_2$, the resulting colon ideal is $(I^2 : g_2) = I + (x_1x_3x_4^2)$. Note that the generator $x_1x_3x_4^2$ has support on three variables, which makes it difficult to compute the depth of this ideal. However, if we multiply $g_2$ by $x_3$, the colon ideal becomes more tractable, and its depth is easier to compute. We therefore adjust the definition of the product of edges as follows.
\end{Remark}

\begin{Definition}\label{def5}
Let $P_{\bf w}^{n+1}$ be an increasing weighted path. For each $i\in[n-1]$, let $f_i = (x_ix_{i+1})^{{\bf w}_i}$. We define the modified monomial $\tilde{f}_i$ by
\[
\tilde{f}_i =
\begin{cases}
f_i, & \text{if ${\bf w}_i = {\bf w}_{i+1}$},\\
f_i \cdot \prod_{j=i+1}^\ell x_j^{{\bf w}_j - 1}, & \text{otherwise},
\end{cases}
\]
where $\ell$ denotes the largest index such that ${\bf w}_i < {\bf w}_{i+1} < \cdots < {\bf w}_\ell$.
Furthermore, let
\[
\Delta = \bigl\{\, i \in [n-2] \mid {\bf w}_i = {\bf w}_{i+1} \,\bigr\}.
\]
Let $\mu_1, \ldots, \mu_{|\Delta|}$ be the labeling of the elements of $\Delta$ as defined above. For each integer $t$ with $2 \le t \le |\Delta|+1$, we define the monomial $g_t \in S$ by
\[
g_t = \prod_{i=1}^{t-1} \tilde{f}_{\mu_i+1}.
\]
\end{Definition}

\begin{Remark}
Continuing with the example in Remark \ref{rem1}, we see that $g_3 = (x_2x_3^2) \cdot (x_{10}x_{11})^6$ and $g_4 = (x_2x_3^2) \cdot (x_{10}x_{11})^6 \cdot \bigl((x_8x_9)^4 x_9^5\bigr)$.
\end{Remark}

To compute the depth of the quotient ring of higher powers of $I(P_{\bf w}^{n+1})$, we use techniques derived from \cite{HHV} and \cite{LVZ} to compute the colon ideal.
For a trivially weighted path, the colon ideal is obtained by bipartite completion of the closed neighborhood of edges. For a subset $\{e_i \mid i \in A\}$ of edges of an increasing weighted path $P_{\bf w}^{n+1}$, we have the monomial $f_i$ and the modified monomial $\tilde{f}_i$. In the following, we compute the colon ideals of products of these $\tilde{f}_i$. For convenience, we introduce some notation:

For a monomial $f$, by abuse of notation, we denote $\min(f) = \min(\supp(f))$ and $\max(f) = \max(\supp(f))$, 
where $\supp(f) = \{i \mid x_i \text{ divides } f\}$. For a subset $A \subseteq [n]$, we define $\odd(A)$ as follows: if there are no odd integers in $A$, then $\odd(A) = 1$; otherwise, $\odd(A)$ is the maximal odd integer in $A$. Similarly, we define $\even(A)$ as follows: if there are no even integers in $A$, then $\even(A) = 0$; otherwise, $\even(A)$ is the maximal even integer in $A$.

For a monomial $D = \prod\limits_{j \in A} \tilde{f}_j$, we denote $\bar{D} = \prod\limits_{j \in A} f_j$, $\esupp(D) = \supp(D/\bar{D})$, $\odd(D) = \odd(\esupp(D))$, and $\even(D) = \even(\esupp(D))$.

Let $D$ be a product of modified monomials such that $\supp(D)$ consists of consecutive integers. We define the following ideals:
\begin{align*}
	U(D) & = \bigl(x_j^{a_j} \mid j \text{ is odd and } j \in [\min(D) - 1, \max(D) + 1]\bigr),\\
	V(D) & = \bigl(x_j^{a_j} \mid j \text{ is even and } j \in [\min(D) - 1, \max(D) + 1]\bigr),
\end{align*}
where
\[
a_j = 
\begin{cases}
	1, & \text{if } j \in \esupp(D),\\
	{\bf w}_j, & \text{if } j \in \bigl([\min(D)-1, \max(D)+1] \cap [n]\bigr) \setminus \esupp(D),\\
	{\bf w}_n, & \text{if } j = n+1.
\end{cases}
\]
We also define
\[
E(D) = \bigl(x_{j}^{a_{j}} \mid j \in [\min(D)-1, \odd(D)-1] \text{ and } j \text{ is even}\bigr),
\]
and
\[
O(D) = \bigl(x_{j}^{a_j} \mid j \in [\min(D)-1, \even(D)-1] \text{ and } j \text{ is odd}\bigr).
\]

\medskip
We illustrate this definition with the following three examples.

\begin{Example}
    Consider the increasing edge-weighted path $P_{\bf w}^7$ corresponding to ${\bf w} = (1,1,2,2,2,3)$. From the definition following Theorem \ref{monoincrease2} and Definition \ref{def5}, we have $A(\Delta) = \{4\}$, $B(\Delta) = \{1,3\}$, and $C(\Delta) = \emptyset$. Choosing $D = g_2 = (x_5x_6)^2x_6^2$, we obtain $E(D) = (0)$, $O(D) = (x_5^2)$, $U(D) = (x_5^2,x_7^3)$, and $V(D) = (x_4^2,x_6)$.
\end{Example}

\begin{Example}
	Consider the increasing edge-weighted path $P_{\bf w}^9$ corresponding to ${\bf w} = (1,1,1,1,3,3,3,4)$. Then $A(\Delta) = \emptyset$, $B(\Delta) = \{5,6\}$, and $C(\Delta) = \{1,2,3\}$. Choosing $D = g_3 = (x_6x_7)^3 (x_7x_8)^3 x_8^3$, we get $E(D) = (0)$, $O(D) = (x_5^3,x_7^3)$, $U(D) = (x_5^3,x_7^3,x_9^4)$, and $V(D) = (x_6^3,x_8)$.
\end{Example}

\begin{Example}
	Consider the increasing edge-weighted path $P_{\bf w}^{12}$ corresponding to ${\bf w} = (1,1,1,1,1,2,2,2,2,2,3)$. Then $A(\Delta) = \emptyset$, $B(\Delta) = \{4,6\}$, and $C(\Delta) = \{9,8,7,3,2,1\}$. Choosing $D = g_6 = (x_5x_6)x_6(x_7x_8)^2(x_8x_9)^2(x_9x_{10})^2(x_{10}x_{11})^2x_{11}^2$, we have $E(D) = (x_4, x_6,x_8^2,x_{10}^2)$, $O(D) = (x_5)$, $U(D) = (x_5,x_7^2,x_9^2,x_{11})$, and 
	$V(D) = (x_4,x_6,x_8^2,x_{10}^2,x_{12}^3)$.
\end{Example}

\begin{Lemma}\label{lem:colon_product_edges}
Let $P_{\bf w}^{n+1}$ be an increasing weighted path, and let $g_t$ be defined as in Definition \ref{def5}. Then $g_t$ can be written as $g_t = \prod\limits_{i=1}^{t_k} D_i$, 
where the support of each $D_i$ forms a maximal block of $\supp(g_t)$.
For each $t \in [2, |\Delta|+1]$, we have
\[
\bigl(I(P_{\bf w}^{n+1})^{t} : g_t\bigr) = I + \sum_{i=1}^{t_k} \Gamma_i,
\]
where $\Gamma_i=E(D_i) + O(D_i) + (U(D_i) \cap V(D_i))$.
\end{Lemma}
\begin{proof}
Let $I = I(P_{\bf w}^{n+1})$. Then $I \subseteq (I^t:g_t)$ by the choice of $g_t$. To prove that the right-hand side is contained in the left-hand side, it suffices to show that for any $i \in [t_k]$, we have $E(D_i) \subseteq (I^t:g_t)$, $O(D_i) \subseteq (I^t:g_t)$, and $U(D_i) \cap V(D_i) \subseteq (I^t:g_t)$. We only prove $E(D_i) \subseteq (I^t:g_t)$ and $U(D_i) \cap V(D_i) \subseteq (I^t:g_t)$, as the inclusion $O(D_i) \subseteq (I^t:g_t)$ can be shown similarly.

\medskip
First, we prove $E(D_i) \subseteq (I^t:g_t)$. Let $x_p^{a_p} \in E(D_i)$, then $p$ is even and $p \in [\min(D_i)-1, \odd(D_i)-1]$. Choose $q$ to be the smallest odd number such that $q \in \esupp(D_i)$ and $p < q \le \odd(D_i)$. If $q = p+1$, then ${\bf w}_p < {\bf w}_{p+1}$. By the definitions of $g_t$ and $a_p$, $x_{p+1}^{{\bf w}_p}\bar{g_t}$ divides $g_t$, and $x_p^{{\bf w}_p - a_p}x_{p+1}^{{\bf w}_p}\bar{g_t} \mid g_t$. Thus, $x_p^{a_p}g_t = \left(f_p\bar{g_t}\right) \cdot \frac{x_p^{a_p}g_t}{x_p^{{\bf w}_p}x_{p+1}^{{\bf w}_p}\bar{g_t}} \in I^t$, where $f_p = x_p^{{\bf w}_p}x_{p+1}^{{\bf w}_p}$.

\medskip
Now, suppose $q > p+1$. We consider the following two cases:

(I) If ${\bf w}_p = {\bf w}_{p+1} = \cdots = {\bf w}_{q-1}$, we set $\Delta_p = \bigl\{\,j \in [p, q-1] \mid {\bf w}_j = {\bf w}_{j+1} \,\bigr\}$ and $\Delta_i = \bigl\{\,j \in [\min(D_i)-1, \max(D_i)-2] \mid {\bf w}_j = {\bf w}_{j+1} \,\bigr\}$. Then $\Delta_p = [p, q-2]$, $\Delta_p \subseteq \Delta_i$, and $\Delta_p \subseteq \Delta$. From the definition of $\bar{D_i}$, we know that $\prod\limits_{j=p+1}^{q-1}f_j$ divides $\bar{D_i}$, and $\bar{D_i}$ divides $\bar{g_t}$, where $\bar{g_t} = \prod\limits_{j=1}^{t-1} f_{\mu_j+1}$. Therefore, $\prod\limits_{j=p+1}^{q-1}f_j$ divides $\bar{g_t}$. Note that ${\bf w}_{q-1} < {\bf w}_q$, so $x_q^{{\bf w}_{q-1}}\bar{g_t}$ divides $g_t$. Furthermore, by the definition of $a_p$, $x_p^{{\bf w}_p - a_p}x_q^{{\bf w}_{q-1}}\bar{g_t} \mid g_t$. A simple calculation shows that
\[
x_p^{a_p}g_t = f_p\bigl(\prod\limits_{j=1}^{\frac{q-p-1}{2}}f_{p+2j}^2\bigr)\bigl(\frac{\bar{g_t}}{\prod_{j=p+1}^{q-1}f_j}\bigr) \cdot \frac{g_t}{x_p^{{\bf w}_p - a_p}x_q^{{\bf w}_{q-1}}\bar{g_t}} \in I^t.
\]

(II) If there exists some $r \in [p+1, q-2]$ such that ${\bf w}_r < {\bf w}_{r+1}$, we may assume that
${\bf w}_p = \cdots = {\bf w}_{r_1} < {\bf w}_{r_1+1} = \cdots = {\bf w}_{r_1+r_2} < {\bf w}_{r_1+r_2+1} = \cdots = {\bf w}_{r_1+r_2+r_3} < \cdots < {\bf w}_{r_1+r_2+\cdots+r_\ell+1} = \cdots = {\bf w}_{q-1}$. Let $\delta_j = \sum_{i=1}^j r_i$ for all $j \in [1, \ell]$. For each induced subpath 	$P_{\bf w}^{[p,\delta_1]}, P_{\bf w}^{[\delta_1+1, \delta_2]},\ldots, P_{\bf w}^{[\delta_\ell+1, q]}$,  the weight function  is a constant,
by similar arguments as (I), we can obtain that 
$x_p^{{\bf w}_p}x_{\delta_1}^{{\bf w}_{\delta_1}}(\prod\limits_{j=p+1}^{\delta_1-1}f_j) \in I^{\delta_1 - p}$, $x_{\delta_\ell+1}^{{\bf w}_{\delta_\ell+1}}x_q^{{\bf w}_{q-1}}(\prod\limits_{j=\delta_\ell+2}^{q-1}f_j) \in I^{q - \delta_\ell - 1}$, and $x_{\delta_{j-1}+1}^{{\bf w}_{\delta_{j-1}+1}}x_{\delta_j}^{{\bf w}_{\delta_j}}(\prod\limits_{m=\delta_{j-1}+2}^{\delta_j-1}f_m) \in I^{r_j - 1}$ for all $j \in [2, \ell]$. Furthermore, $h \mid \bar{D}_i$ and $\bar{D}_i \mid \bar{g}_t$, where
\begin{align*}
	h &= \left(\prod\limits_{j=p+1}^{\delta_1} f_j \right)
	\left(\prod\limits_{\alpha=2}^{\ell} \prod\limits_{\beta=\delta_{\alpha-1}+2}^{\delta_\alpha} f_\beta \right)
	\left(\prod\limits_{j=\delta_\ell+2}^{q-1} f_j \right) \in I^{q - p - \ell - 1},
\end{align*}
which implies $\bar{g}_t/h \in I^{t - q + p + \ell}$. By the definition of $a_p$, $x_p^{{\bf w}_p - a_p}\bar{g}_t \mid g_t$. Since ${\bf w}_{\delta_j} < {\bf w}_{\delta_j+1}$ for all $j \in [\ell]$ and ${\bf w}_{q-1} < {\bf w}_q$, we have $h_1\bar{g}_t \mid g_t$, where $h_1 = (x_p^{{\bf w}_p - a_p}x_q^{{\bf w}_{q-1}}) \cdot \prod\limits_{j=1}^{\ell} x_{\delta_j+1}^{{\bf w}_{\delta_j+1} - {\bf w}_{\delta_j}}$. Therefore,
\begin{align*}
	x_p^{a_p}g_t &= \left((hx_p^{{\bf w}_p}x_q^{{\bf w}_{q-1}}) \cdot \prod\limits_{j=1}^{\ell} x_{\delta_j+1}^{{\bf w}_{\delta_j+1} - {\bf w}_{\delta_j}} \right)\left(\frac{\bar{g}_t}{h}\right)\left(\frac{g_t}{h_1\bar{g}_t}\right) \\
	&= \left((x_p^{{\bf w}_p}x_{\delta_1}^{{\bf w}_{\delta_1}}) \cdot \prod\limits_{j=p+1}^{\delta_1-1}f_j\right) \cdot \left(\prod\limits_{j=2}^{\ell}\Bigl(x_{\delta_{j-1}+1}^{{\bf w}_{\delta_{j-1}+1}}x_{\delta_j}^{{\bf w}_{\delta_j}} \prod\limits_{m=\delta_{j-1}+2}^{\delta_j-1}f_m\Bigr)\right) \cdot \\
	&\quad \left(x_{\delta_\ell+1}^{{\bf w}_{\delta_\ell+1}}x_q^{{\bf w}_{q-1}} \prod\limits_{j=\delta_\ell+2}^{q-1}f_j\right) \left(\frac{\bar{g}_t}{h}\right) \cdot \left(\frac{g_t}{h_1\bar{g}_t}\right) \\
	&\in I^{\delta_1 - p} \cdot I^{q - \delta_\ell - 1} \cdot I^{\delta_\ell - \delta_1 - \ell + 1} \cdot I^{t - q + p + \ell} = I^t.
\end{align*}

Next, we prove $U(D_i) \cap V(D_i) \subseteq (I^t:g_t)$.

By the definitions of $U(D_i)$ and $V(D_i)$, it suffices to consider products of the form $x_\alpha^{a_\alpha} x_\beta^{a_\beta}$, where $\alpha < \beta$ and the factors are chosen from $U(D_i) \setminus O(D_i)$ and $V(D_i) \setminus E(D_i)$, respectively. If ${\bf w}_\alpha = \cdots = {\bf w}_\beta$, then by similar arguments to those in Case (I), $\bar{D}_i \mid \bar{g}_t$ and $\prod\limits_{j=\alpha}^{\beta-2}f_j \mid \bar{D_i}$. Furthermore, by the definitions of $a_\alpha$ and $a_\beta$, $x_{\alpha}^{{\bf w}_\alpha - a_\alpha}x_{\beta}^{{\bf w}_\beta - a_\beta}\bar{g}_t \mid g_t$. Therefore,
\begin{align*}
x_{\alpha}^{a_{\alpha}}x_{\beta}^{a_{\beta}}g_t &= (f_{\alpha}f_{\beta-1}) \cdot \bigl(\prod_{j=1}^{\frac{\beta - \alpha - 3}{2}}f_{\alpha+2j}^2\bigr)\left(\frac{\bar{g}_t}{\prod_{j=\alpha+1}^{\beta-2}f_j}\right) \cdot \left(\frac{g_t}{x_{\alpha}^{{\bf w}_\alpha - a_\alpha}x_{\beta}^{{\bf w}_\beta - a_\beta}\bar{g}_t}\right) \\
&\in I^{2 + (\beta - \alpha - 3) + (t - \beta + \alpha + 1)} = I^t.
\end{align*}

If ${\bf w}_\alpha < {\bf w}_\beta$, then there exists an $m \in [\alpha+1, \beta-1]$ such that ${\bf w}_m < {\bf w}_{m+1}$. By similar arguments to those in Case (II), we can deduce that 
$x_{\alpha}^{a_{\alpha}}x_{\beta}^{a_{\beta}}g_t \in I^t$.

To end the proof, we show the opposite inclusion.  Let $h$ be a minimal generator of $(I^t:g_t)$. Then $h = \frac{\prod\limits_{j=1}^{t}f_{i_j}}{\gcd\left(\prod\limits_{j=1}^{t}f_{i_j},g_t\right)}$. Without loss of generality, we may assume that $\min(f_{i_1}) \le \min(f_{i_2}) \le \cdots \le \min(f_{i_t})$.
We will prove that $h \in \Theta$, where $\Theta = I + \sum\limits_{i=1}^{t_k} \Gamma_i$, by induction on $n \ge 3$ and $t \ge 2$. The case $n = 3$ is trivial. Now, suppose $n \ge 4$. We first consider the case $t = 2$.

In this case, let $g_2 = \tilde{f}_\mu$ for some $\mu \in [2, n-1]$, then ${\bf w}_{\mu-1} = {\bf w}_\mu$ and the expression of $h$ simplifies to $h = \frac{f_{i_1}f_{i_2}}{\gcd(f_{i_1}f_{i_2},\tilde{f}_\mu)}$. If $\supp(f_{i_j}) \cap \supp(g_2) = \emptyset$ for some $j \in [2]$, then $h \in (f_{i_j}) \subseteq I$. Otherwise, we consider the following two cases:

(I) If ${\bf w}_{\mu} = {\bf w}_{\mu+1}$, then $g_2 = \tilde{f}_\mu = f_\mu = (x_\mu x_{\mu+1})^{{\bf w}_\mu}$. A simple calculation shows that
\[
h \in \left(x_{\mu-1}^{{\bf w}_{\mu-1}}x_\mu^{{\bf w}_{\mu}}, x_{\mu-1}^{{\bf w}_{\mu-1}}x_{\mu+2}^{{\bf w}_{\mu+2}}, x_{\mu}^{{\bf w}_{\mu}}x_{\mu+1}^{{\bf w}_{\mu+1}}, x_{\mu+1}^{{\bf w}_{\mu+1}}x_{\mu+2}^{{\bf w}_{\mu+2}}\right) \subseteq U(D_1) \cap V(D_1)\subseteq \Theta.
\]

(II) If ${\bf w}_{\mu} < {\bf w}_{\mu+1}$, then $g_2 = \tilde{f}_\mu = f_\mu = (x_\mu x_{\mu+1})^{{\bf w}_\mu} \prod\limits_{j=\mu+1}^{\lambda} x_j^{{\bf w}_j - 1}$, where $\lambda$ is the largest index such that ${\bf w}_\mu < {\bf w}_{\mu+1} < \cdots < {\bf w}_\lambda$. Furthermore, the indices $i_1$ and $i_2$ in the expression of $h$ satisfy $i_1, i_2 \in [\mu-1, \lambda]$.

If $\lambda = \mu+1$, then
$h \in \left(x_{\mu}^{{\bf w}_{\mu}}, x_{\mu-1}^{{\bf w}_{\mu-1}}x_{\mu+2}^{{\bf w}_{\mu+2}}, x_{\mu+1}x_{\mu+2}^{{\bf w}_{\mu+2}}\right) \subseteq \Gamma_1\subseteq \Theta$.
Otherwise, $h \in \left(x_{\mu-1}^{{\bf w}_{\mu-1}}, x_{\mu+1}, \ldots, x_{\lambda-1}, x_\lambda x_{\lambda+1}^{{\bf w}_{\lambda+1}}\right) \subseteq \Gamma_1\subseteq \Theta$.

Now, assume $t \ge 3$. We consider four cases:

(1) If $\min(f_{i_1}) = 1$ and $\min(g_t) = 2$, then $x_1$ does not divide $\gcd(\prod\limits_{j=1}^{t}f_{i_j},g_t)$. Thus, we can write $g_t = x_2^{{\bf w}_2}\left(\tilde{f}_2/x_2^{{\bf w}_2}\right) \cdot g_{t-1}$ and $h = x_1^{{\bf w}_1} m$, where  $m=\frac{ \prod\limits_{j=2}^{t}f_{i_j}}{\gcd ( \prod\limits_{j=2}^{t}f_{i_j},(\tilde f_2/x_2^{{\bf w}_2}) g_{t-1})}$. It follows that $m \in (I^{t-1} : g_{t-1}(\tilde f_2/x_2^{{\bf w}_2}))$. By the induction hypothesis, $\left(I^{t-1} : g_{t-1}\right) = I + L + K$, where $L = E(D_1/\tilde{f}_2) + O(D_1/\tilde{f}_2) + (U(D_1/\tilde{f}_2) \cap V(D_1/\tilde{f}_2))$ and $K = \sum\limits_{i=2}^{t_k}\Gamma_i$.

Note that  $\supp(K)\cap\supp(\tilde f_2/x_2^{{\bf w}_2})=\emptyset$ and  $\supp(L)\cap\supp(\tilde f_2/x_2^{{\bf w}_2})=\{\max(\tilde f_2)\}$. We have
\begin{align*}
x_1^{{\bf w}_1}(K:(\tilde f_2/x_2^{{\bf w}_2})) &= x_1^{{\bf w}_1}K \subseteq \Theta, \quad \text{and} \\
x_1^{{\bf w}_1}(L:(\tilde f_2/x_2^{{\bf w}_2})) &=
\begin{cases}
x_1^{{\bf w}_1}\left(L : x_{3}^{{\bf w}_{3}}\right), & \text{if $\max(\tilde{f}_2) = 3$ and ${\bf w}_2 = {\bf w}_3$}, \\
x_1^{{\bf w}_1}\left(L : x_{3}^{{\bf w}_{2} + {\bf w}_{3} - 1}\right), & \text{if $\max(\tilde{f}_2) = 3$ and ${\bf w}_2 < {\bf w}_3$}, \\
x_1^{{\bf w}_1}\left(L : x_{\max(\tilde{f}_2)}^{{\bf w}_{\max(\tilde{f}_2)} - 1}\right), & \text{if $\max(\tilde{f}_2) \ge 4$}.
\end{cases}
\end{align*}
Since $x_1^{{\bf w}_1} \in O(D_1)$ if and only if $x_3^{a_3} \in O(D_1)$, we have  $x_1^{{\bf w}_1}(L:(\tilde f_2/x_2^{{\bf w}_2}))\subseteq  \Gamma_1\subseteq\Theta$. On the other hand, it is easy to verify that $(I(P_{\bf w}^{[2,n+1]}): (\tilde f_2/f_2)) \subseteq I +\Gamma_1 \subseteq \Theta$. Therefore, $x_1^{{\bf w}_1} (I^{t-1} : g_{t-1}(\tilde f_2/x_2^{{\bf w}_2}))\subseteq \Theta$, which implies $h \in \Theta$.

(2) If $\min(f_{i_1}) = 1$ and $\min(g_t) \ge 3$, then $(I^t:g_t) \subseteq I$, since $f_1 \mid h$.

(3) If $\min(f_{i_1}) \ge 2$ and $\min(g_t) = 2$, we can write $g_t = f_2 \cdot (\tilde{f}_2 / f_2) \cdot g_{t-1}$. Thus, by Lemma \ref{lem_colon_leaf},  $h\in ((I(P_{\bf w}^{[2,n+1]})^{t-1} : g_{t-1}) : (\tilde f_2 / f_2))$. By induction on $n$ and $t$,
\[
\left(I(P_{\bf w}^{[2,n+1]})^{t-1} : g_{t-1}\right) = I(P_{\bf w}^{[2,n+1]}) + L + K.
\]
Using arguments similar to those for $K$ and $L$ in Case (1), we can deduce that $(K:(\tilde f_2/f_2))\subseteq  \Theta$ and $(L:(\tilde f_2/f_2))\subseteq  \Theta$. It is also easy to check that  $(I(P_{\bf w}^{[2,n+1]}): (\tilde f_2/f_2))\subseteq I+E(D_1) + O(D_1) + (U(D_1) \cap V(D_1))\subseteq \Theta$. Therefore,$((I(P_{\bf w}^{[2,n+1]})^{t-1} : g_{t-1}) : (\tilde f_2 / f_2))\subseteq \Theta$, which implies $h \in \Theta$.

(4) If $\min(f_{i_1}) \ge 2$ and $\min(g_t) \ge 3$, then by induction on $n$,  $h\in (I(P_{\bf w}^{[2,n+1]})^{t} : g_{t})=I(P_{\bf w}^{[2,n+1]}) + \sum_{i=1}^{t_k}\Gamma_i\subseteq  \Theta$.
\end{proof}

Given an increasing weighted path $P_{\bf w}^{n+1}$,  we set  $f_i = (x_ix_{i+1})^{{\bf w}_i}$  for  $i\in[n-1]$, and $\Delta = \bigl\{\, i \in [n-2] \mid {\bf w}_i = {\bf w}_{i+1} \,\bigr\}$.
We can rewrite set $B(\Delta)$ as  $B(\Delta)=\{b_1,b_2,\ldots,b_{2b(\Delta)-1},b_{2b(\Delta)}\}$, where $b_1<b_2<\cdots<b_{2b(\Delta)}$. In particular, if $b(\Delta)=0$, then $B(\Delta)=\emptyset$.
As stated in Definition \ref{def5}, we can define the monomial $g_t=\prod\limits_{i=1}^{t-1} \tilde{f}_{\mu_i+1}$.
 For each  $j\in[n+1]\setminus\esupp(g_t)$, we define a function $\eta_j$ as follows:
\[
\eta_j=
\begin{cases}
	{\bf w}_{j-1}-1, &\text{if $j=n+1$},\\
{\bf w}_{j-1}, &\text{if   $j=b_{2i-1}+1$  such that  $\tilde{f}_{b_{2i}+1}\nmid g_t$},\\
	{\bf w}_{j-2}-1, &\text{if  $j=b_{2i-1}+2$ such that  $\tilde{f}_{b_{2i}+1}\nmid g_t$},\\
	{\bf w}_j-1, &\text{otherwise}.
\end{cases}
\]

\begin{Lemma}\label{colon}
	For each $t \in [2, |\Delta|+1]$, let  $g_t$ be defined as in Definition \ref{def5}. We write $g_t$ in the form  $g_t = \prod\limits_{i=1}^{t_k} D_i$ as in Lemma \ref{lem:colon_product_edges}.
Let $\Lambda =[n+1]\setminus\esupp(g_t)$ and $\rho_t=g_t\prod\limits_{j\in \Lambda}x_j^{\eta_j}$. Then
\[
(I(P_{\bf w}^{n+1})^t:\rho_t)=I(P_{n+1}) +\Phi_t+\Psi+ \sum_{i=1}^{t_k} \Upsilon_i,
\]
	where $P_{n+1}$ is a trivially weighted path, $\Phi_t=(\{x_{b_{2i-1}},x_{b_{2i-1}+2}\}:\tilde{f}_{b_{2i}+1}\nmid g_t \text{\ with }i\in[b(\Delta)])$,
 $\Psi=(x_j\mid {\bf w}_j<{\bf w}_{j+1} \text{ and }  j\neq b_{2i-1}+1 \text{ for any } i\in[b(\Delta)])$ and $\Upsilon_i=E'(D_i) + O'(D_i) + (U'(D_i) \cap V'(D_i))$ with
 $E'(D_i)=\bigl(x_j\mid j \in  \supp(E(D_i))\bigr)$, $O'(D_i)=\bigl(x_j\mid j \in  \supp(O(D_i))\bigr)$, $U'(D_i)=\bigl(x_j\mid j \in  \supp(U(D_i))\bigr)$ and  $V'(D_i)=\bigl(x_j\mid j \in  \supp(V(D_i))\bigr)$.
\end{Lemma}
\begin{proof}
	Let $I=I(P_{\bf w}^{n+1})$.  First,  we   prove that the right-hand side is contained in the left-hand side.  Since $\rho_t=g_t\prod\limits_{j\in \Lambda}x_j^{\eta_j}=(\bar{g_t}\prod\limits_{j\in \esupp(g_t)}x_j^{{\bf w}_j-1})(\prod\limits_{j\in \Lambda }x_j^{\eta_j})$, it is easy to verify that 
\[
	(I(P_{n+1}) +\Phi_t+\Psi)\subseteq \Bigl(I:(\prod\limits_{j\in \esupp(g_t)}x_j^{{\bf w}_j-1})\prod\limits_{j\in \Lambda}x_j^{\eta_j}\Bigr)\subseteq (I^t:\rho_t).
\]
On the other hand, 	from the beginning  of the proof of Lemma \ref{lem:colon_product_edges}, we know that for any $i\in [t_k]$ we have $\Gamma_i\subseteq (I^{t} : g_t)$.
For any  $h\in \Upsilon_i$, by the definition of $\eta_j$, $E(D_i),  O(D_i),U(D_i)$ and $V(D_i)$, we can deduce that
 $h\prod\limits_{j\in \Lambda}x_j^{\eta_j}\in \Gamma_i\subseteq(I^t:g_t)$. So $h\in (I^t:\rho_t)$.

Next,  we   prove that the left-hand side is contained in the right-hand side. By Lemma \ref{lem:colon_product_edges},  $(I^t:g_t)=I + \sum\limits_{i=1}^{t_k} \Gamma_i$. Moreover, $(I:\prod\limits_{j\in \Lambda}x_j^{\eta_j})\subseteq I(P_{n+1})+\Phi_t+\Psi$   is trivial.
For any $i\in [t_k]$, since $\supp(\Gamma_i)=[\min(D_i)-1,\max(D_i)+1]$,  
$(\Gamma_i:\prod\limits_{j\in\Lambda}x_j^{\eta_j})=(\Gamma_i:\prod\limits_{j\in\Lambda_1}x_j^{\eta_j})$,
where $\Lambda_1=[\min(D_i)-1,\max(D_i)+1]\setminus\esupp(D_i)$. Note that we always have $\max(D_i)\in \Delta\cup\{n-1,n\}$. By the definition of $\eta_j$, 
$(\Gamma_i:\prod\limits_{j\in \Lambda_1}x_j^{\eta_j})\subseteq \Upsilon_i+\Psi$ holds.
 We complete the proof. 
\end{proof}

 Given an increasing weighted path $P_{\bf w}^{n+1}$,  we set  $f_i = (x_ix_{i+1})^{{\bf w}_i}$  for  $i\in[n-1]$, and $\Delta = \bigl\{\, i \in [n-2] \mid {\bf w}_i = {\bf w}_{i+1} \,\bigr\}$.
We can also rewrite set $C(\Delta)$ as  $C(\Delta)=\{c_1,c_2,\ldots,c_{3c(\Delta)}\}$, where $c_1<c_2<\cdots<c_{3c(\Delta)}$. In particular, if $c(\Delta)=0$, then $C(\Delta)=\emptyset$. 
Now, we will prove the main theorem.

\begin{proof}[Proof of Theorem \ref{maintheorem}] Let $I=I(P_{\bf w}^{n+1})$.  The notation  $\rho_t$, $\Lambda$, $E'(D_i)$, $O'(D_i)$, $U'(D_i)$ and $V'(D_i)$ are defined as  in Lemma \ref{colon}. 
Since $\rho_t\notin I^t$, we would like  to prove that $\depth (S /(I^t: \rho_t)) \le d(\Delta,t)$. The result follows from Lemmas \ref{lowerbounds} and \ref{radup}.
There are three cases:

(1) If $t \in[2, a(\Delta)+1]$, then $\tilde{f}_{b_{2i}+1}\nmid g_t$  for any $i\in[b(\Delta)]$.
By Lemma \ref{colon},  $(I^t:\rho_t)=I(P_{n+1})+ \Phi_t +\Psi+\sum\limits_{i=1}^{t_k} \Upsilon_i$ 	and $\Phi_t=(\{x_{b_{2i-1}},x_{b_{2i-1}+2}\}:i\in[b(\Delta)])$. 
Easy to verify that
\[
[n+1]\setminus\Omega=\left(\supp(I^t:\rho_t)\setminus\Omega\right) \sqcup \left((\bigcup\limits_{j=1}^{b(\Delta)}\{b_{2j-1}+1\})\setminus\bigcup\limits_{i=1}^{t_k}\{\max(D_i)+1\}\right)\sqcup \Pi,
\]
where $\Omega=\bigcup\limits_{i=1}^{t_k}[\min(D_i)-1,\max(D_i)+1]$ and   $\Pi=\{j\in A(\Delta): \tilde{f}_{j+1}\nmid g_t \text{ and } j-1\in  [n+1]\setminus C(\Delta) \}$. 
Furthermore,
\[
\Omega=\left(\supp(I^t:\rho_t)\cap\Omega\right)  \sqcup \left((\bigcup\limits_{j=1}^{b(\Delta)}\{b_{2j-1}+1\})\cap(\bigcup\limits_{i=1}^{t_k}\{\max(D_i)+1\})\right).
\]
Therefore, 
\[
[n+1]=\supp(I^t:\rho_t)\sqcup  (\bigcup\limits_{j=1}^{b(\Delta)}\{b_{2j-1}+1\})\sqcup   \Pi.
\]
Since $|(\bigcup\limits_{j=1}^{b(\Delta)}\{b_{2j-1}+1\})\sqcup\Pi|=b(\Delta)+|\Pi|$, $\depth(S/(I^t:\rho_t))=b(\Delta)+|\Pi|+\depth(S_1/H_1)$, where $S_1=\KK[x_i\mid i\in\supp(H_1)]$, and
$H_1=\bigl(\sum\limits_{i\in M_t} J_i\bigr)+K+I(P')$ 
with $M_t=\bigl\{i\in[t_k]\mid \max(D_i)\in C(\Delta)\sqcup\{n-1,n\}\bigr\}$, $J_i=(U'(D_i)\setminus O'(D_i)) \cap (V'(D_i)\setminus E'(D_i))$, 
\[
K=
\begin{cases}
	(x_nx_{n+1}), &\text{if ${\bf w}_{n-1}<{\bf w}_n$ or $n=b_{2b(\Delta)-1}+3$},\\
	(x_{n-1}x_n,x_nx_{n+1}), &\text{otherwise},
\end{cases}
\]
 and $P'$ is an induced subpath of $P_{n+1}$ on the set $C(\Delta)\sqcup\Pi_1$, where $\Pi_1=\{j\in A(\Delta): \tilde{f}_{j+1}\nmid g_t \text{ and } j\in C(\Delta)\}$.

To compute $\depth(S/(I^t:\rho_t))$,  it is enough to 	compute $\depth(S_1/H_1)$.
Applying \cite[Lemma 2.10]{LVZ} to the edge $e_{c_1}$, we have 
\[
\depth(S_1/H_1)\le
\begin{cases}
	\depth(S_{11}/H_{11})+1, &\text{if $c_1+4\notin \Pi_1$},\\
	\depth(S_{12}/H_{12})+2, &\text{if  $c_1+4\in \Pi_1$},
\end{cases}
\]
where each $S_{1j}=\KK[x_q\mid q\in\supp(H_{1j})]$,
\[
H_{1j}=
\begin{cases}
	\bigl(\sum\limits_{i\in M_t} J_i\bigr)+K+I(P'_j), &\text{if  $\max(D_j)\neq c_1$ for all $j\in M_t$},\\
	\bigl(\sum\limits_{i\in M_t\setminus\{\ell\}} J_i\bigr)+K+I(P'_j), &\text{if $\max(D_\ell)=c_1$ for some $\ell\in M_t$},\\
\end{cases}
\]
$P'_1$ is an induced subpath of $P_{n+1}$ on the set $(C(\Delta)\setminus\{c_1,c_2,c_3\})\sqcup\Pi_1$ and $P'_2$ is an induced subpath of $P_{n+1}$ on the set $(C(\Delta)\setminus\{c_1,c_2,c_3\})\sqcup(\Pi_1\setminus\{c_1+4\})$.

Now we need to compute $\depth(S_{1j}/H_{1j})$ for $j\in[2]$. Applying \cite[Lemma 2.10]{LVZ} successively to the edges $e_{c_4},e_{c_7},\ldots,e_{c_{3c(\Delta)-5}},e_{c_{3c(\Delta)-2}}$ yields
$\depth(S_{11}/H_{11})\le c(\Delta)+|\Pi_1|$ and $\depth(S_{12}/H_{12})\le c(\Delta)+|\Pi_1|-1$.
Note that from  the definition of $\Pi_1$, we have  $|\Pi_1|=a(\Delta)-(t-1)-| \Pi|$.
Therefore, 
 $\depth(S_1/H_1)\le c(\Delta)+|\Pi_1|+1=c(\Delta)+a(\Delta)-|\Pi|-t+2$. Thus, $\depth (S/(I^t:\rho_t))\le  k(\Delta)-t+2= d(\Delta,t)$.

(2) If $t\in  [a(\Delta)+1, a(\Delta) + 2b(\Delta)+1]$, then   $d(\Delta,a(\Delta)+2j+1)= d(\Delta,a(\Delta)+2j+2)$  for any $0\le j\le b(\Delta)-1$. 
Then by \cite[Lemma 2.6]{HLTZ}, we only need to consider the case where  $t = a(\Delta) + 2u+1$ with $0\le u\le b(\Delta)$. The case $u=0$ follows from (1). In the following, we assume that $u\in[b(\Delta)]$. 
By Lemma \ref{colon},  $(I^t:\rho_t)=I(P_{n+1})+ \Phi_t +\Psi+\sum\limits_{i=1}^{t_k} \Upsilon_i$ and $\Phi_t=(\{x_{b_{2i-1}},x_{b_{2i-1}+2}\}:i\in[b(\Delta)-u])$.
Easy to check that
$[n+1]\setminus\Omega=\left(\supp(I^t:\rho_t)\setminus\Omega\right)\sqcup \left(\{b_{2i-1}+1\mid i\in[b(\Delta)-u]\}\right)$ and 
$\Omega\subseteq \supp(I^t:\rho_t)$.
Therefore, 
$[n+1]=\supp(I^t:\rho_t)\sqcup  \left(\{b_{2i-1}+1\mid i\in[b(\Delta)-u]\}\right)$.
Since $|\{b_{2i-1}+1\mid i\in[b(\Delta)-u]\}|=b(\Delta)-u$, $\depth(S/(I^t:\rho_t))=b(\Delta)-u+\depth(S_2/H_2)$, where $S_2=\KK[x_i\mid i\in\supp(H_2)]$,
\begin{align*}
		H_2=\bigl(\sum\limits_{i\in [t_k]} J_i\bigr)+I(P'')+
		\begin{cases}
			(x_nx_{n+1}), &\text{if ${\bf w}_{n-1}<{\bf w}_n$},\\
			(x_{n-1}x_n,x_nx_{n+1}), &\text{if ${\bf w}_{n-1}={\bf w}_n$},
		\end{cases}
	\end{align*}
 and $P''$ is an induced subgraph of $P_{n+1}$ on the set $C(\Delta)$. By similar arguments as the case (1), we can deduce that  $\depth(S_2/H_2)\le c(\Delta)+1$. Therefore, $\depth (S/(I^t:\rho_t))\le b(\Delta)-u+c(\Delta)+1= d(\Delta,t)$.

(3) If $t\in  [a(\Delta) + 2b(\Delta)+1, |\Delta|]$, then we also have that $d(\Delta,a(\Delta) + 2b(\Delta)+3j+1)=d(\Delta,a(\Delta) + 2b(\Delta)+3j+2)=d(\Delta,a(\Delta) + 2b(\Delta)+3j+3)$  for any $0\le j\le c(\Delta)-1$. 
As a result, we only need to consider the case where  $t = a(\Delta) + 2b(\Delta)+3v+1$ with $0\le v\le c(\Delta)-1$. The case $v=0$ follows from (2). In the following, we assume that $v\in[c(\Delta)-1]$. In this case,  $(I^t:\rho_t)=I(P_{n+1}) +\Psi+\sum\limits_{i=1}^{t_k} \Upsilon_i$ and $\Phi_t=(0)$. 
		By the definition of $g_t$ and $\rho_t$, we have 
	 $\tilde{f}_{j+1}\mid \rho_t$  for any $j\in A(\Delta)\cup B(\Delta)\cup\{c_{3(c(\Delta)-v)+1},c_{3(c(\Delta)-v)+2},\ldots,c_{3c(\Delta)}\}$. Therefore,  $(I^t:\rho_t)=\Psi+\sum\limits_{i=1}^{t_k}(O'(D_i)+E'(D_i))+H_3$, where 
	\begin{align*}
		H_3=\bigl(\sum\limits_{i\in [t_k]} J_i\bigr)+I(P''')+
		\begin{cases}
			(x_nx_{n+1}), &\text{if ${\bf w}_{n-1}<{\bf w}_n$},\\
			(x_{n-1}x_n,x_nx_{n+1}), &\text{if ${\bf w}_{n-1}={\bf w}_n$},
		\end{cases}
	\end{align*}
and	$P'''$ is an induced subgraph of $P_{n+1}$ on the set $\{c_1,\ldots,c_{3(c(\Delta)-v)}\}$. 
	Since the ideal $\Psi+\sum\limits_{i=1}^{t_k}(O'(D_i)+E'(D_i))$ is generated by variables,   $\depth(S/(I^t:\rho_t))=\depth(S_3/H_3)$, where $S_3=\KK[x_i\mid i\in\supp(H_3)]$. By similar arguments as the case (1), we can deduce that  $\depth (S/(I^t:\rho_t))=\depth(S_3/H_3)\le c(\Delta)-v+1= d(\Delta,t)$. 
 \end{proof}

\medskip
\hspace{-6mm} {\bf Acknowledgment}

 \vspace{3mm}
\hspace{-6mm}  This research is supported by  the National Natural Science Foundation of China (No.
12471246). The authors are grateful to the software systems \cite{Co} and \cite{GS}
 for providing us with a large number of examples.

\medskip
\hspace{-6mm} {\bf Data availability statement}

\vspace{3mm}
\hspace{-6mm}  The data used to support the findings of this study are included within the article.

\medskip
\hspace{-6mm} {\bf Conflict of interest}

\vspace{3mm}
\hspace{-6mm}  The authors declare that they have no competing interests.

	\end{document}